\documentclass[12pt]{article}
\usepackage[utf8]{inputenc}
\usepackage[english]{babel}
\usepackage{amsthm,amsmath,amsfonts}
\usepackage{hyperref}

\usepackage{currfile}
\usepackage{navigator}
\IfFileExists{\currfilename}{\embeddedfile{sourcefile}{\currfilename}}{}
\IfFileExists{\currfilebase.bbl}{\embeddedfile{bibliography}{\currfilebase.bbl}}{} 
 \hypersetup{pdftitle={Multifold 1-perfect codes},
              pdfsubject={completely regular codes},
              pdfauthor={Denis Krotov}}

\newtheorem{proposition}{Proposition}
\newtheorem{theorem}{Theorem}
\newtheorem{lemma}{Lemma}
\newtheorem{corollary}{Corollary}

\theoremstyle{remark}
\newtheorem{remark}{Remark}

\makeatletter
\def\@seccntformat#1{\csname the#1\endcsname.\ } 
\makeatother

\newcommand\mspan[1]{\langle\!\langle #1 \rangle\!\rangle}
\newcommand\Mspan[1]{\mspan{#1}^*}
\newcommand\GF[1]{\mathbb{F}_{#1}}
\newcommand\0{\overline 0}

\title{Multifold $1$-Perfect Codes%
\thanks{The work was funded by
the Russian Science Foundation, Grant 22-11-00266, \linebreak[5]\url{https://rscf.ru/en/project/22-11-00266/}}
} 
\author{Denis S. Krotov%
\thanks{Sobolev Institute of Mathematics, Novosibirsk, Russia. E-mail: \texttt{dk@ieee.org}}
} 
\date{}

\begin{document}

\maketitle

\begin{abstract}
A multifold $1$-perfect code ($1$-perfect code for list decoding) in any graph is a set $C$ of vertices such that every vertex of the graph is at distance not more than $1$ from exactly $\mu$ elements of $C$. In $q$-ary Hamming graphs, where $q$ is a prime power, we characterize all parameters of multifold $1$-perfect codes and all parameters of additive multifold $1$-perfect codes. In particular, we show that additive multifold $1$-perfect codes are related to special multiset generalizations of spreads, multispreads, and that multispreads of parameters corresponding to multifold $1$-perfect codes always exist.

\end{abstract}

\section{Introduction}
Perfect codes play a special role
in coding theory and discrete mathematics.
They are optimal
as error-correcting codes,
as covering codes, as orthogonal arrays,
and as representatives 
of other classes of objects
related to different applications or
theoretical problems.
Multifold perfect codes generalize 
perfect codes and 
are still optimal as, 
for example, 
list-decoding codes
(see, e.g.,~\cite{Guruswami:05}, \cite{Resch:20})
or multifold
ball packings~\cite{KroPot:multifold},
multiple covering
codes~\cite{CHLL}. 

In the current paper,
we characterize
the parameters
of multifold $1$-perfect codes
in $q$-ary Hamming graphs,
where~$q$ is a prime power.
Before, such characterization
was known for the case when $q$ is prime
and for general $q$ but only for linear codes,
see~\cite[Section~14.2]{CHLL}.
As we will see, for non-prime~$q$,
not all parameters of multifold $1$-perfect codes
can be represented by linear or by the union of cosets of linear multifold $1$-perfect codes
(the first counterexample is a $5$-fold $1$-perfect $4$-ary code of length~$13$),
and the technique related to additive codes essentially
extends the class of such codes and allows to describe their parameters.
Before giving a brief survey 
of related results and formulating
the characterization theorem, 
we will define the main concepts 
such as multifold perfect codes and
Hamming graphs.

\paragraph{Main definitions.}
A set~$C$ of vertices of 
a graph
is called a \emph{$\mu$-fold} (\emph{multifold})
\emph{$e$-perfect code}
if every radius-$e$ ball
(i.e., the set of vertices 
at distance not more than $e$ 
from some fixed vertex, 
the center of the ball) 
contains exactly $\mu$ 
elements of~$C$.
To avoid trivial cases,
we also imply that $C$ 
is not empty
and does not coincide with the graph vertex set.
$1$-fold $e$-perfect codes
are called \emph{$e$-perfect codes},
or just \emph{perfect codes}.
In regular graphs,
perfect codes are a special case 
of completely regular codes 
defined in the next paragraph.

 An \emph{equitable 
 $l$-partition} of a graph
is an ordered partition $(C_1, \ldots, C_l)$ 
of its vertex set into~$l$ 
subsets, \emph{cells},
such that for every~$i$
and~$j$ from~$1$ to~$l$ 
every vertex in~$C_i$ 
has exactly~$S_{i,j}$
neighbors in~$C_j$, 
where $S_{i,j}$ depends on~$i$ and~$j$
but does not depend on the choice of the vertex in~$C_i$;
the $l$-by-$l$ matrix $(S_{i,j})_{i,j=1}^l$
is called the \emph{quotient matrix} of the equitable partition
(in some literature, the intersection matrix or just the parameter matrix).
A set~$C$ (code) of vertices of a graph
is called
\emph{completely regular}
if the distance partition $(C^{(0)},C^{(1)},\ldots,C^{(\rho)})$
with respect to~$C$
(i.e., $C^{(i)}$ is the set of vertices
at distance~$i$ from $C$)
is equitable.
The number~$\rho$ of 
the cells different from~$C$ 
in the distance partition 
is called the 
\emph{covering radius} of~$C$.
The quotient matrix corresponding
to a completely regular code 
is always tridiagonal because, obviously,
$C^{(i)}$ can have neighbors only in
$C^{(i-1)}$, $C^{(i)}$, and $C^{(i+1)}$.

\begin{remark}
 Completely regular codes with covering radius~$1$ are also known as \emph{intriguing sets}~\cite{DeBruSuz:2010}.
\end{remark}

The Hamming graph $H(n,q)$ is a graph
the vertices of which are 
the words of length~$n$ over an alphabet
of cardinality~$q$, 
two words being adjacent 
if and only if they differ 
in exactly one position.
If~$q$ is a prime power,
then the alphabet is often considered as 
the finite field~$\GF{q}$,
and the set of vertices of~$H(n,q)$
forms an $n$-dimensional 
vector space~$\GF{q}^n$
over~$\GF{q}$, 
with coordinate-wise addition 
and multiplication 
by a constant.
The subspaces of~$\GF{q}^n$
are called \emph{linear codes},
while the subgroups 
of the additive group of~$\GF{q}^n$
are known as \emph{additive codes}.
Alternatively, additive codes can
be treated as $\GF{p}$-linear subspaces,
where~$p$ is prime and $q=p^t$,
because $\GF{q}^n$ is a $tn$-dimensional vector space
over the order-$p$ subfield of~$\GF{q}$.
In the current paper, 
we mainly consider this last 
representation of the additive codes.
 
\paragraph{A mini-survey.} The problem of complete
characterization of parameters
of perfect codes 
in the Hamming graphs
$H(n,q)$ is solved only
for the case when
$q$ is a prime power~\cite{ZL:1973}, 
\cite{Tiet:1973}: there are no
nontrivial perfect codes except the
$e$-perfect repetition codes
in
$H(2e + 1, 2)$, the $3$-and
$2$-perfect Golay codes~\cite{Golay:49} in
$H(23, 2)$ and $H(11, 3)$, respectively, and the
$1$-perfect codes
in
$H((q^k-1)/(q-1),q)$. In the case of a non-prime-power $q$, no nontrivial perfect codes are known, and the parameters
for which the nonexistence is not proven 
are restricted by
$1$- and $2$-perfect codes 
(the last case is solved 
for some values of $q$), 
see~\cite{Heden:2010:non-prime} 
for a survey of known results 
in this area.

Let us consider the multifold case.
The existence 
of $\mu$-fold $e$-perfect codes
with $e$ larger than~$1$
is a complicated problem,
especially for small~$\mu$.
A survey can be found 
in~\cite[Section~14.2]{CHLL};
in~\cite{Vorobev:2012:multiple},
Vorob'ev considered
multifold $r$-perfect codes
in $H(n,2)$, $r>1$,
related to equitable
$2$-partitions.
In this paper, we focus on the case $r=1$.

For a $\mu$-fold $1$-perfect code $C$ in $H(n,q)$, 
 the code and its complement form an equitable partition with quotient matrix
$$
\left(
\begin{array}{cc}
 \mu-1 & (q-1)n-\mu+1 \\
 \mu & (q-1)n-\mu \\
\end{array}
\right).
$$
The famous Lloyd condition (see, e.g., \cite[Theorem~9.3.3]{GoRo} for equitable partitions) says that the eigenvalue $-1$ of this quotient matrix must belong to the eigenspectrum
$\{-n+iq :\ i\in\{0,\ldots,n\}\}$ of the 
adjacency matrix of the
Hamming graph $H(n,q)$,
which is equivalent to $n\equiv 1\bmod q$
(for example, by this reason
$5$-fold $1$-perfect codes
of size $2^5$ 
in $H(3,4)$ cannot exist).
Existence 
of such partitions 
in the binary case 
was shown 
in~\cite{FDF:PerfCol}, 
as a part of 
a general theory
of equitable $2$-partitions of $H(n,2)$ 
(which was partially generalized to an arbitrary~$q$ in~\cite{BKMTV}), but the sufficiency of the sphere-packing 
condition 
for the existence 
of multifold $1$-perfect codes 
was known earlier
for any prime $q$:
\begin{proposition}[{\cite[Theorem~14.2.4]{CHLL}}]\label{p:perfect}
Assume that $q$ is prime. A $\mu$-fold
$1$-perfect code of cardinality~$K$ in $H(n,q)$
exists if and only if $K=\mu q^n/|B|$ 
is integer and $\mu\le |B|$, 
where~$B$ is a radius-$1$ ball in $H(n,q)$, 
$|B|=(q-1)n+1$.
\end{proposition}

The sufficiency in Proposition~\ref{p:perfect} is proved by
constructing a code as the union of~$\kappa$ cosets 
of a linear $\mu/\kappa$-fold 
$1$-perfect code 
of dimension~
$k$
(defined, for example, by a parity-check matrix,
generalizing the Hamming code \cite[p.193]{MWS}), 
where
\begin{equation}\label{eq:kK}
 K=\frac{\mu q^n}{(q-1)n+1}=\kappa q^k, \qquad \gcd(\kappa,q)=1.
\end{equation}
The same argument works to construct 
$\mu$-fold $1$-perfect codes in $H(n,q)$ for any
prime-power~$q$; so, \eqref{eq:kK} is sufficient
for the existence, but not necessary: if~$q$ is not prime,
then not all integer~$K$ are representable in the form~\eqref{eq:kK} with 
$\gcd(\kappa,q)=1$. 
The smallest example (also satisfying the Lloyd condition) is $q=4$, $n=13$, $\mu=5$. In this case,
$K=2^{23}=2\cdot 4^{11}$, 
and we see that $\mu/\kappa$ is not integer
for any representation of~$K$ in the form~$\kappa 4^k$.
Therefore, a code cannot be constructed 
as the union of cosets of 
a linear multifold $1$-perfect code.
However, a $5$-fold 
$1$-perfect code in $H(13,4)$
can be constructed 
as an additive 
code, and this particular example was resolved in~\cite[Sect.\,V.B]{KroPot:multifold}.

\paragraph{Outline.}
In the current note, 
we construct 
multifold $1$-perfect codes
of any admissible parameters 
in Hamming graph~$H(n,q)$,
where~$q$ is a prime power.
In Section~\ref{s:spread},
we describe additive
completely regular codes
with covering radius~$1$
in terms of so-called multispreads,
multiset analogs of spreads,
and list some constructions
of multispreads, 
utilized in the following section.
In Section~\ref{s:perfect},
we characterize 
the parameters
of additive multifold 
$1$-perfect codes 
and show that a
multifold 
$1$-perfect code
of any admissible parameters
can be constructed
as the union of cosets of 
an additive multifold 
$1$-perfect code.

\section
[Multispreads and additive completely regular codes of covering radius 1]
{Multispreads and additive completely regular codes of covering radius $1$}
\label{s:spread}
Let $p$ be a prime power, 
$t$ and $m$ positive integers,
$q=p^t$.

The Hamming graph~$H(n,q)$ is
a graph on the set~$\Sigma^n$
of words of length~$n$
($n$-words)
over an alphabet~$\Sigma$ 
of size~$q$; 
two $n$-words are adjacent
if they differ in exactly 
one position.
We define~$\Sigma$ as the set of
$t$-words over~$\GF{p}$.
The elements of~$\Sigma^n$ 
can then be treated as $nt$-words
over~$\GF{p}$, 
where each $nt$-word
is the concatenation of~$n$ words of 
length~$t$, called \emph{blocks}.
With such treatment, 
two $nt$-words are adjacent 
in $H(n,q)$ if and only if 
they differ
in exactly one block 
(independently on the number 
of differing positions 
in that block).

The vertices of~$H(n,q)$,
with coordinate-wise addition
and multiplication by a constant
from~$\GF{p}$, 
form an $nt$-dimensional space
over~$\GF{p}$.
Any subspace of this space is called
an \emph{$\GF{p}$-linear code} in~$H(n,q)$.

\begin{remark}
In contrast,
a \emph{linear code} in~$H(n,q)$
is a subspace over 
the larger field~$\GF{q}$,
where the elements of~$H(n,q)$
are treated as $n$-words 
over~$\GF{q}$.
Since~$\GF{q}$ is itself 
a $t$-dimensional vector space 
over~$\GF{p}$, its elements
can be written as $t$-tuples
over~$\GF{p}$, in some fixed basis.
This shows that every linear code
is $\GF{p}$-linear. 
The inverse is not true if $q>p$;
in particular, 
the cardinality of an $\GF{p}$-linear code
is a power of~$p$
but is not necessarily a power of~$q$.
\end{remark}

Every subspace~$C$ of the space of
$nt$-words over~$\GF{p}$
can be represented as the kernel,
$\ker(M)$,
of an $(nt-\dim C) \times nt$
matrix~$M$, 
called a 
\emph{check matrix} 
of the $\GF{p}$-linear code~$C$.
Since we consider the positions
of $nt$-words as arranged into
blocks, the columns of a check
matrix are also naturally grouped
into $n$ groups.
For a collection 
$T_1$, \ldots,  $T_n$ 
of $t$-tuples
of vectors from $\GF{p}^m$,
we denote by 
$M(T_1, \ldots, T_n)$
the $m \times nt$ matrix 
with columns
$T_1^1$, \ldots,  $T_1^t$, $T_2^1$, \ldots,  $T_n^t$,
where $T_i=(T_i^1, \ldots, T_i^t)$.

For a given finite multiset~$V$ of vectors,
by $\mspan{V}$ we denote the multiset
$ \{\sum_{v \in V} a_v v:\ a_v\in \GF{p} \} $
of all~$p^{|V|}$ linear combinations
of elements from~$V$;
by $\Mspan{V}$ we denote the multiset
$ \{\sum_{v \in V} a_v v:\ a_v\in \GF{p} \} 
- \{\0\} $
of  all $p^{|V|}-1$
nontrivial linear combinations
(with at least one coefficient 
being nonzero)
of elements from~$V$.
Every such $\mspan{V}$
will be called a
\emph{multisubspace}
(of the vector space),
or \emph{$|V|$-multisubspace},
with a basis~$V$.

We will call a collection 
$(S_1$, \ldots,  $S_n)$ of $t$-multisubspaces
of   $\GF{p}^m$ 
a \emph{$(\lambda,\mu)$-spread},
or \emph{multispread},
 if there holds
$$
{S_1} \uplus \ldots \uplus {S_n} 
=
(n+\lambda){\times} \{\0\}
 \uplus 
 \mu{\times} \GF{p}^{m*},
$$
or, equivalently,
\begin{equation}\label{eq:lm}
{S_1}^* \uplus \ldots \uplus {S_n}^* 
=
\lambda{\times} \{\0\}
 \uplus 
 \mu{\times} \GF{p}^{m*},
 \qquad\mbox{where }
 {S_i}^*=S_i - \{\overline 0\}.
\end{equation}
We note that the definition is
`additive' in the sense that the union of a $(\lambda,\mu)$-spread and a $(\lambda',\mu')$-spread is a $(\lambda+\lambda',\mu+\mu')$-spread.
Multispreads with parameters $(0,1)$
are known as \emph{spreads}, while
more general
$(0,\mu)$-spreads are known as 
\emph{$\mu$-fold spreads};
in that case, $S_i$
are ordinary $t$-dimensional subspaces,
without multiplicity larger than~$1$.

Our first result is almost trivial.
Since it is also very important for us,
we give a detailed proof.
\begin{theorem}\label{th:iff}
Assume that 
 $T_1$, \ldots, $T_n$ are $t$-tuples
of vectors from~$\GF{p}^m$ and all the vectors
from $T_1$, \ldots, $T_n$ span~$\GF{p}^m$.
The code $\ker M$, 
$M=M(T_1, \ldots, T_n)$,
is an $\GF{p}$-linear completely regular code 
in $H(n,p^t)$ 
with quotient matrix
\begin{equation}\label{eq:LaMu}
\left(
\begin{array}{cc}
\lambda  & n(q-1)-\lambda \\
\mu     & n(q-1)-\mu 
\end{array}
\right)
\end{equation}
if and only if
$(\mspan{T_1},\ldots,\mspan{T_n})$
is a $(\lambda,\mu)$-spread.
\end{theorem}
\begin{proof}
By the definition,
every codeword~$x$ 
satisfies $Mx =0$.
Every word at distance~$1$
from~$x$ has the form
$x-e$, where~$e$ has nonzeros
in exactly one, say $i$th,
block. 
If $x-e$ is also a codeword,
then $Me =0$, which means that
the linear combination 
of the vectors from~$T_i$
with the coefficients
from the $i$th block of~$e$ equals~$0$.
Therefore, 
\begin{itemize}
\item[(a)] the number~$\lambda$
of codeword neighbors of~$x$
is the number of times
$0$ occurs 
as a nontrivial combination 
of vectors from some~$T_i$,
$i=1,\ldots,n$:
\begin{equation}
\label{eq:La}
\lambda = \sum_{i=1}^n
\Big|\big\{(a_v)_{v\in T_i}:\ \sum_{v\in T_i} a_v v = 0, \mbox{ not all $a_v$ are zero}\big\}\Big|;
\end{equation}
note that the right part is the multiplicity of~$0$ in the multiset \linebreak[5]
$\mspan{T_1}^*\uplus\ldots\uplus\mspan{T_n}^*$.
\end{itemize}

Next, for a non-codeword~$y$,
we have $My =s_y$
for some non-zero~$s_y$.
For every codeword~$y-e$
at distance~$1$ from~$y$,
it holds $Me =s_y$.
Similarly to the arguments above,
\begin{itemize}
\item[(b)] the number~$\mu$
of codeword neighbors of~$y$
is the number of times
$s_y$ occurs
as a nontrivial combination 
of vectors from some~$T_i$,
$i=1,\ldots,n$:
\begin{equation}
\label{eq:Mu}
\mu = \sum_{i=1}^n
\Big|\big\{(a_v)_{v\in T_i}:\ \sum_{v\in T_i} a_v v = s_y\big\}\Big|;
\end{equation}
note that the right part is the multiplicity of~$s_y$ in the multiset \linebreak[5]
$\mspan{T_1}^*\uplus\ldots\uplus\mspan{T_n}^*$.
\end{itemize}
The claim of the theorem 
is now straightforward
from~(a), (b), and 
definitions.
Indeed, by the definition
\begin{itemize}
 \item[(*)] the code and its complement form an equitable
$2$-partition with quotient matrix~\eqref{eq:LaMu} (equivalently, the code is completely regular with quotient matrix~\eqref{eq:LaMu})
\end{itemize}
if and only if
\begin{itemize}
 \item every codeword~$x$ has exactly~$\lambda$
code neighbors (which constant does not depend on~$x$
because the code is additive) and
every non-codeword~$y$  has exactly~$\mu$ code neighbors,
where~$\mu$ does not depens on the choice of~$y$.
\end{itemize}
Since by the hypothesis of the theorem every
element of~$\GF{p}^m$
is represented as~$My$ for some~$y$,
from~\eqref{eq:La} and~\eqref{eq:Mu}
we see that (*) is equivalent to
\begin{itemize}
 \item $(\mspan{T_1},\ldots,\mspan{T_n})$
is a $(\lambda,\mu)$-spread,
\end{itemize}
by the definition. This completes the proof.
\end{proof}

A construction of multifold spreads
can be found in
\cite[p.83]{Hirschfeld79}.
We 
recall the following
two special cases:

\begin{lemma}\label{l:multispread}
If $t\le m$, 
then there is a 
$\frac{p^t-1}{p-1}$-fold spread
from $t$-dimensional subspaces
of~$\GF{p}^{m}$.
\end{lemma}
\begin{proof}
Since $t\le m$, 
there is a collection $Y=(y_1,\ldots,y_t)$ of elements
of~$\GF{p^{m}}$ that are linearly independent over~$\GF{p}$.
Let $v_1$, \ldots, $v_{\frac{p^m-1}{p-1}}$ 
be a collection of
$|\GF{p}^{m\,*}|/|\GF{p}^*|
=\frac{p^m-1}{p-1}$ pairwise 
$\GF{p}$-independent elements
in~$\GF{p^{m}}$.
It is easy to see that 
the collection
$(Y_1$, \ldots,
$Y_{\frac{p^m-1}{p-1}})$,
where $Y_i = \langle v_i y_1,\ldots,v_i y_t \rangle$ (the span is over~$\GF{p}$),
is a $\frac{p^t-1}{p-1}$-fold spread from $t$-subspaces
of~$\GF{p^{m}}$, considered as an $m$-dimensional vector space over~$\GF{p}$.
\end{proof}

\begin{lemma}\label{l:spread}
If $t$ divides~$m'$,
then there is a spread from $\frac{p^{m'}-1}{p^t-1}$
$t$-dimensional subspaces of~$\GF{p}^{m'}$.
\end{lemma}
\begin{proof}
An example of such a spread
is the partition
of the field~$\GF{p^{m'}}$ 
into multiplicative cosets 
of the subfield~$\GF{p^t}$.
\end{proof}

\begin{corollary}\label{c:spread-1}
If $t$ divides $m+s$, 
then there is a $(p^s-1,p^s)$-spread
 from 
 $\frac{p^{m+s}-1}{p^t-1}$
 $t$-dimensional multisubspaces of~$\GF{p}^{m}$.
\end{corollary}
\begin{proof}
Consider a spread from Lemma~\ref{l:spread},
where $m'=m+s$.
Projecting along $s$ coordinates
turns it into 
a required $(p^s-1,p^s)$-spread.
\end{proof}

Additionally, 
an $(\alpha(p^t-1),0)$-spread, $r\in\{0,1,\ldots\}$,
can be constructed
as 
$\alpha$ copies of the trivial
multisubspace
$p^t{\times}\{0\} = \{0\}\uplus\ldots\uplus \{0\}$.
We summarize the constructions
above in the following proposition.
\begin{proposition}\label{p:sum}
Let $s$ be the smallest 
nonnegative integer 
such that $t$ divides $m+s$ (i.e., $-m\equiv s \bmod t$).
For all nonnegative integer $\alpha$, $\beta$, and $\gamma$,
\begin{itemize}
    \item[\rm(a)]
    there exists an
    $\big(\alpha(p^t-1)+\beta(p^s-1),\beta p^s \big)$-spread;
    \item[\rm(b)] if $t\le m$,
    then
    there is an
    $\big(\alpha(p^t-1)+\beta(p^s-1),\beta p^s + \gamma\frac{p^t-1}{p-1} \big)$-spread.
\end{itemize}
\end{proposition}
\begin{proof}
A required multispread is constructed as the union
of $\alpha$ copies of the
$(p^t-1,0)$-spread,
$\beta$ copies of a spread 
from Corollary~\ref{c:spread-1},
and, if $t\le m$, $\gamma$ copies
of a spread from Lemma~\ref{l:multispread}.
\end{proof}

\section
[Multifold 1-perfect codes]
{Multifold $1$-perfect codes}
\label{s:perfect}
In this section, we prove 
two main theorems of the paper,
which characterize the parameters
of additive and, respectively,
unrestricted multifold
$1$-perfect codes 
in Hamming graphs $H(n,q)$,
where $q$ is a prime power.
The claim of Theorem~\ref{th:add}
is formulated in a more
general form, 
for $\GF{p}$-linear codes,
where additive codes 
correspond to the case of prime~$p$.

\begin{theorem}\label{th:add}
Let $q=p^t$ for a prime power~$p$.
An $\GF{p}$-linear 
$\mu$-fold
$1$-perfect code 
in~$H(n,q)$ exists 
if and only if for some
integer $k<nt$ and $m=nt-k$ there hold
\begin{itemize}
    \item[\rm(i)] 
    $
    \displaystyle
\frac{\mu q^n}{(q-1)n+1} = p^k,
$
\item[\rm(ii)] $m\ge t$ or 
$p^{t-m}$ divides $\mu$. 
\end{itemize}
Moreover, if {\rm(i)} holds, then {\rm(ii)}  is equivalent to 
\begin{itemize}
\item[\rm(ii')] $n \equiv 1 \bmod q$. 
\end{itemize}
\end{theorem}
\begin{proof}
\emph{Necessity of {\rm(i)}.} 
If there is an 
$\GF{p}$-linear $\mu$-fold 
$1$-perfect code~$C$,
then, by the definition, each of~$q^n$
vertices of~$H(n,q)$ 
is at distance at most one from exactly 
$\mu$ codewords of~$C$. On the other hand,
each of~$|C|$ codewords has exactly 
$(q-1)n+1$ vertices 
at distance at most~$1$. 
By doublecounting, we get
$$ \mu q^n = ((q-1)n+1)|C|. $$
If $C$ is $\GF{p}$-linear, then $|C|=p^k$,
where~$k$ is its $p$-dimension. (i) follows.

\emph{Necessity of {\rm(ii)}.}
Each $\GF{p}$-linear code of $p$-dimension~$k$
is the kernel of an~$m\times nt$ matrix
$M(T_1, ..., T_n)$.
If the code is $\mu$-fold $1$-perfect,
then
by Theorem~\ref{th:iff}
$(\mspan{T_1}, \ldots , \mspan{T_n})$
is a $(\mu-1,\mu)$-multispread.
If, additionally, $t>m$, then
each $\mspan{T_i}$ is a multiset
with multiplicities~$p^{t-r_i}$,
where~$r_i$ is the rank 
of the corresponding submatrix.
Since $r_i\le m$, we see that
all the multiplicities, and hence~$\mu$,
are divisible by~$p^{t-m}$.

\emph{Equivalence of {\rm(ii)} and {\rm(ii')}.} Rewriting (i) as
$$ \mu p^m = q n-(n-1), $$ 
we see that $n\equiv 1 \bmod q$ 
if and only if $ \mu p^m $ is divisible by $q=p^t$. The last is obviously equivalent to (ii).

\emph{Sufficiency of {\rm(i)+(ii)}.}
We rewrite {\rm(i)} in the following form:
\begin{equation}\label{eq:pm1}
\mu p^m - 1 = n(p^t-1),  
\end{equation}

Firstly, consider the case $m\ge t$.
Denote by $s$ the number
from $\{0,\ldots,t-1\}$
such that $m+s$
is divisible by $t$.
We have
$$p^m  =
(p^{m+s-t}-1)p^{t-s} + p^{t-s}
\equiv p^{t-s} \bmod (p^t-1),$$
and
$$\mu p^m -1 \equiv
\mu p^{t-s} -1 \bmod (p^t-1).$$
From~\eqref{eq:pm1} we see that the
left part is divisible by $p^t-1$; 
hence,
\begin{equation}\label{eq:mut}
\mu p^{t-s} -1 \equiv 0
\bmod (p^t-1).
\end{equation}
The first $\mu$ that meets
\eqref{eq:mut} is $p^{s}$;
the general case is 
$\mu = p^{s} +\alpha (p^t-1)$,
$\alpha \in \{0,1,\ldots\}$.
A $(\mu-1,\mu)$-spread
is provided by Proposition~\ref{p:sum}(b),
where $\beta=1$ and 
$\gamma = (p-1)\alpha$.

It remains to consider the case
$m < t$. By {\rm(ii)},
$\mu/p^{t-m}$ is integer, 
and a
$(\mu-1,\mu)$-spread
is provided by
Proposition~\ref{p:sum}(a),
where $\beta = \mu/p^{t-m}$
and 
$\alpha = 
({\frac{\mu}{p^{t-m}} - 1})
/ (p^{t}-1)$.
The integrity of
$\alpha$
is derived from~\eqref{eq:mut},
where $t-s=m$: indeed,
$\frac{\mu}{p^{t-m}}p^t-1
\equiv 0 \bmod (p^t-1)$ 
implies  
$\frac{\mu}{p^{t-m}} 
\equiv 1 \bmod (p^t-1)$.
\end{proof}

\begin{theorem}\label{th:all}
Let $q=p^t$ for a prime $p$.
A $\mu$-fold $1$-perfect code 
in $H(n,q)$ exists 
if and only if 
\begin{itemize}
\item[\rm(i)]
$\mu < (q-1)n+1$,
\item[\rm(ii)] 
$\displaystyle
\frac{\mu q^n}{(q-1)n+1}
$
is integer, 
\item[\rm(iii)] and $n \equiv 1 \bmod q$. 
\end{itemize}
\end{theorem}
\begin{proof}
The necessity
of (i) and (ii)
is obvious and (iii)
is Lloyd's condition.

To show the sufficiency,
we assume that
$\displaystyle
\frac{\mu q^n}{(q-1)n+1}=\kappa p^k
$
for some integer $k$ and $\kappa$ such that
$\gcd(\kappa,p)=1$.
We see that $\mu$
is divisible by $\kappa$,
and by Theorem~\ref{th:add}
we can construct an additive 
$\frac\mu\kappa$-fold
$1$-perfect code~$C$.
Then,
the union of~$\kappa$ cosets
of~$C$
is a required $\mu$-fold
$1$-perfect code.
\end{proof}

\section{Conclusion}
Additive codes are natural algebraic generalization of linear codes
over $\GF{q}$ in the case when
$q$ is a prime power but not a prime.
In the current note, we have shown
that a multifold $1$-perfect code                                                                                    
of any admissible parameters can be constructed
as the union of cosets of an additive multifold $1$-perfect code.
Recently~\cite{SKS:drg},
another putative parameter set
of completely regular codes was resolved
with an additive code over~$\GF4$,
while the generalization of that code to 
a putative infinite series remains an open problem.
Additive codes have attracted a lot of attention in recent years, and, in particular,
the study of additive completely regular codes
becomes an interesting and promising line of research.
Some natural directions are characterization of parameters of multispreads
and constructing the geometrical 
theory of additive completely
regular codes of covering radius~$2$, similar
to~\cite{CalKan:86:2weight} for linear codes.

\subsection*{Declaration of competing interest, data availability}
The author declares that he has no known competing financial interests
or personal relationships that could have appeared to influence the work
reported in this paper.

No data was used or generated during the research described in the article.

%

\providecommand\href[2]{#2} \providecommand\url[1]{\href{#1}{#1}}
  \def\DOI#1{{\small {DOI}:
  \href{http://dx.doi.org/#1}{#1}}}\def\DOIURL#1#2{{\small{DOI}:
  \href{http://dx.doi.org/#2}{#1}}}

\end{document}